\documentclass[11pt]{amsart}

\usepackage[pdftex]{graphicx}
\usepackage{amssymb}
\usepackage{amsmath}
\usepackage{amsfonts}
\usepackage{amsthm}
\usepackage{amssymb}
\usepackage{mathrsfs}
\usepackage{hyperref}
\newtheorem{theorem}{Theorem}[section]

\newtheorem{corollary}[theorem]{Corollary}
\newtheorem{claim}[theorem]{Claim}

\newtheorem{lemma}[theorem]{Lemma}

\theoremstyle{definition}

\theoremstyle{remark}
\newtheorem{remark}[theorem]{Remark}

\begin{document}
\title[Scalar curvature and harmonic forms on manifolds with boundary]{Scalar curvature and harmonic one-forms on three-manifolds with boundary}
\author[H. Bray]{Hubert L. Bray}
\address{
Department of Mathematics\\
Duke University\\
Durham, NC, 27708\\
USA}
\email{bray@math.duke.edu}
\author[D. Stern]{Daniel L. Stern}
\address{
Department of Mathematics\\
University of Toronto\\
Toronto, ON, M5S 2E4\\
Canada
}
\email{dl.stern@utoronto.ca}

\maketitle

\begin{abstract} For a homotopically energy-minimizing map $u: N^3\to S^1$ on a compact, oriented $3$-manifold $N$ with boundary, we establish an identity relating the average Euler characteristic of the level sets $u^{-1}\{\theta\}$ to the scalar curvature of $N$ and the mean curvature of the boundary $\partial N$. As an application, we obtain some natural geometric estimates for the Thurston norm on $3$-manifolds with boundary, generalizing results of Kronheimer-Mrowka and the second named author from the closed setting. By combining these techniques with results from minimal surface theory, we obtain moreover a characterization of the Thurston norm via scalar curvature and the harmonic norm for general closed, oriented three-manifolds, extending Kronheimer and Mrowka's characterization for irreducible manifolds to arbitrary topologies.
\end{abstract}

\section{Introduction}

The study of scalar curvature and its relation to the global geometry and topology of Riemannian manifolds is a subject of long-standing interest in differential geometry and general relativity. Historically, two key techniques have been used to relate scalar curvature to the global structure of a manifold--the Dirac operator methods originating in the work of Lichnerowicz \cite{Lich} (and further developed by Hitchin \cite{Hitch}, Witten \cite{Wit}, Gromov-Lawson \cite{GL1}-\cite{GL3}, and others), and the minimal hypersurface methods pioneered by Schoen and Yau \cite{SY.psc.1},\cite{SY.psc.2},\cite{SY.pmt} in the late 1970s. 

In \cite{St19}, the second named author introduced a new tool for relating scalar curvature to the global structure of closed $3$-manifolds, based on the study of harmonic maps to $S^1=\mathbb{R}/\mathbb{Z}$--or equivalently, harmonic one-forms with integral periods. Namely, for the level sets $\Sigma_{\theta}=u^{-1}\{\theta\}$ of a nonconstant harmonic circle-valued map $u: M\to S^1$ on a closed, oriented $3$-manifold, it was shown in \cite{St19} that the average Euler characteristic $\int_{\theta\in S^1}\chi(\Sigma_{\theta})$ satisfies the lower bound
\begin{equation}\label{ds.eq}
2\pi\int_{\theta\in S^1}\chi(\Sigma_{\theta})\geq \int_{\theta\in S^1}\int_{\Sigma_{\theta}}\frac{1}{2}(|du|^{-2}|Hess(u)|^2+R_M),
\end{equation}
where $R_M$ is the scalar curvature of $M$. This identity--which resembles an averaged version of a key identity for stable minimal surfaces (cf. \cite{SY.psc.1})--gives rise to elementary new proofs of several geometric inequalities and rigidity results related to the scalar curvature in dimension three, such as the fact that any metric on $T^3$ of nonnegative scalar curvature must be flat.

In the present paper, we extend these ideas in two significant directions. First, we establish an appealing generalization of the identity \eqref{ds.eq} for homotopically energy-minimizing maps on $3$-manifolds with boundary, which incorporates the \emph{mean curvature} of the boundary in a natural way. This identity recovers in a straightforward way several known results about the structure of compact $3$-manifolds with lower bounds on the interior scalar curvature and boundary mean curvature. Here we emphasize applications to the study of the Thurston norm on compact $3$-manifolds with boundary, generalizing some results of Kronheimer-Mrowka \cite{KM} and the second named author \cite{St19} to manifolds with boundary.

Second, we combine our methods with some celebrated results of Meeks and Yau from minimal surface theory \cite{MY} to improve the results of \cite{St19} in the setting of closed $3$-manifolds. Namely, while the techniques of \cite{St19} are sensitive to the presence of nonseparating minimal spheres in a given closed $3$-manifold $M$, we circumvent this issue by cutting $M$ along a collection of minimal spheres generating the spherical portion of $H_2(M;\mathbb{Z})$, and apply our techniques to the resulting $3$-manifold $N$ with minimal boundary. As a result, we obtain a geometric characterization of the Thurston norm on $H^1(M;\mathbb{Z})$ in terms of the scalar curvature and the harmonic norm on \emph{general closed, oriented $3$-manifolds}, extending the results of \cite{KM} and \cite{St19} to arbitrary topologies.

In what follows, let $(N^3,g)$ be a compact, connected, oriented $3$-manifold with boundary $\partial N$. Recall that for any homotopy class $[u]\in [N:S^1]$ of circle-valued maps on $N$, standard Hodge theory--applied to the cohomology class $[u^*(d\theta)]\in H^1_{dR}(N)$--yields the existence of an energy-minimizing representative $u: N\to S^1=\mathbb{R}/\mathbb{Z}$, unique up to constant rotations. 

It is straightforward to check that a map $u:N\to S^1$ minimizes energy in its homotopy class if and only if the gradient one-form $h=u^*(d\theta)$ is \emph{harmonic}
\begin{equation}\label{harm}
dh=d^*h=0\text{ on }N
\end{equation}
and satisfies the \emph{homogeneous Neumann condition}
\begin{equation}\label{neu}
\langle h,\nu\rangle=0\text{ on }\partial N,
\end{equation}
on the boundary, where $\nu$ denotes the outward unit normal to $\partial N$. In Section \ref{main.comp} below, we refine the arguments of \cite{St19} to obtain the following identity.

\begin{theorem}\label{id.thm} For a harmonic map $u: N^3\to S^1$ with homogeneous Neumann condition, we have the identity
\begin{equation}\label{main.id}
\int_{S^1}\left(\int_{\Sigma_{\theta}}\frac{1}{2}(|du|^{-2}|Hess(u)|^2+R_N)+\int_{\partial\Sigma_{\theta}}H_{\partial N}\right)\leq 2\pi\int_{S^1}\chi(\Sigma_{\theta}),
\end{equation}
where $R_N$ is the scalar curvature of $N$, $H_{\partial N}=tr_{\partial N}(D\nu)$ is the mean curvature of $\partial N$, and $\chi(\Sigma_{\theta})$ is the Euler characteristic of the level set $\Sigma_{\theta}$.
\end{theorem}

\begin{remark} As in \cite{St19}, we remark that each regular level set $\Sigma_{\theta}$ may of course have multiple components, and $\chi(\Sigma_{\theta})$ denotes the sum of their Euler characteristics. 
\end{remark}

Recall that Poincar\'{e}-Lefschetz duality provides an isomorphism 
$$[N:S^1]\cong H^1(N;\mathbb{Z})\cong H_2(N, \partial N;\mathbb{Z}),$$
which may be realized geometrically by associating to any homotopy class $[u]\in [N:S^1]$ the relative homology class of a regular level set $\Sigma_{\theta}=u^{-1}\{\theta\}$. In particular, for any class $\alpha\in H_2(N,\partial N;\mathbb{Z})$, note that the level sets of the energy-minimizing map $u: N\to S^1$ in the dual homotopy class provide a family of distinguished representatives of $\alpha$. As in the closed case, we invite the reader to compare these representatives--arising from the linear pde \eqref{harm}-\eqref{neu}--to the stable free boundary minimal surfaces representing $\alpha$ arising from geometric measure theory.

As our first application of Theorem \ref{id.thm}, we extend to manifolds with boundary some geometric estimates for the Thurston norm on $H^1(N;\mathbb{Z}),$ originally established by Kronheimer and Mrowka for closed irreducible $3$-manifolds \cite{KM}, and extended in \cite{St19} to closed $3$-manifolds with no nonseparating spheres. Recall that the \emph{Thurston norm} of a relative homology class $\alpha\in H_2(N,\partial N;\mathbb{Z})$ was defined in \cite{Th} as the minimum
\begin{equation}\label{th.def}
\|\alpha\|_{Th}:=\min\{\chi_-(\Sigma)\mid [\Sigma]=\alpha\in H_2(N,\partial N)\},
\end{equation}
over embedded surfaces $(\Sigma,\partial \Sigma)\subset (N,\partial N)$ representing $\alpha$, of the quantity
$$\chi_-(\Sigma):=\Sigma_{i=1}^k\max\{0,-\chi(S_i)\}$$
giving the sum of $|\chi(S_i)|$ over the connected components $S_i$ of $\Sigma$ with negative Euler characteristic. For example, if $N\cong S^3\setminus \nu K$ is the manifold with boundary obtained by removing a tubular neighborhood $\nu K$ of a nontrivial knot $K\subset S^3$, then the Thurston norm of a generator $\alpha\in H_2(N,\partial N;\mathbb{Z})\cong \mathbb{Z}$ is given by $\|\alpha\|_{Th}=2g(K)-1$, where $g(K)$ is the genus of the knot.

In view of the relation \eqref{main.id}, it is unsurprising that lower bounds on the Thurston norm of a class $\alpha\in H_2(N,\partial N;\mathbb{Z})$ will force $N$ to have a certain amount of negative scalar curvature or negative mean curvature on the boundary. More precisely, we obtain from Theorem \ref{id.thm} the following estimate, whose proof we discuss in detail in Section \ref{bdry.thurst.sec}.

\begin{theorem}\label{bdry.thurst} Let $(N,g)$ be a compact, oriented $3$-manifold with boundary $\partial N$, such that every connected, embedded surface $(\Sigma,\partial \Sigma)\subset (N,\partial N)$ of genus zero is trivial in $H_2(N,\partial N;\mathbb{Z})$. For any nonzero class $\alpha\in H_2(N,\partial N;\mathbb{Z})$, the energy-minimizing representative $u:N\to S^1$ of the homotopy class dual to $\alpha$ satisfies
\begin{equation}
4\pi\|\alpha\|_{Th}\leq \|du\|_{L^2(N)}\|R_N^-\|_{L^2(N)}+2\|du\|_{L^2(\partial N)}\|H_{\partial N}^-\|_{L^2(\partial N)},
\end{equation}
with equality only if $N$ is covered isometrically by the cylinder $\Sigma \times \mathbb{R}$ over a surface $\Sigma$ of constant nonpositive curvature with (possibly empty) boundary $\partial \Sigma$ of constant nonpositive geodesic curvature.
\end{theorem}

In the closed case, this reduces to Theorem 1.2 of \cite{St19}, stating that on a closed $3$-manifold $M$ with no nonseparating spheres, the Thurston norm satisfies the upper bound
\begin{equation}\label{og.est}
\|\alpha\|_{Th}\leq \frac{1}{4\pi}\|\alpha\|_H\|R_M^-\|_{L^2},
\end{equation}
where $\|\alpha\|_H$ denotes the \emph{harmonic norm}
\begin{equation}
\|\alpha\|_H:=\inf\{\|du\|_{L^2}\mid u\in C^1(N,S^1)\text{ dual to }\alpha\}.
\end{equation}
As an immediate consequence of Theorem \ref{bdry.thurst}, we see that \eqref{og.est} holds for certain three-manifolds with mean-convex boundary $H_{\partial N}\geq 0$.

\begin{corollary}\label{mean.cvx.cor} Let $(N,g)$ is a compact, oriented $3$-manifold with mean convex boundary $H_{\partial N}\geq 0$, such that every connected, embedded surface $(\Sigma,\partial\Sigma)\subset (N,\partial N)$ of genus zero is trivial in $H_2(N,\partial N;\mathbb{Z})$. Then for any nontrivial class $\alpha\in H_2(N,\partial N;\mathbb{Z})$, we have
\begin{equation}
4\pi \|\alpha\|_{Th}\leq \|\alpha\|_H\|R_N^-\|_{L^2(N)},
\end{equation}
with equality only if $N$ is covered isometrically by the cylinder $\Sigma \times \mathbb{R}$ over a surface of constant nonpositive curvature with (possibly empty) geodesic boundary.
\end{corollary}

\begin{remark} One immediate consequence of the rigidity statement in Corollary \ref{mean.cvx.cor} is the nonexistence of metrics with $R_N>0$ and $H_{\partial N}\geq 0$ on compact, oriented $3$-manifolds $N$ with $H_2(N,\partial N;\mathbb{Z})\neq 0$ that contain no nonseparating spheres or disks (relative to $\partial N$). For instance, we immediately recover the fact that the complement $S^3\setminus \nu K$ of a nontrivial knot $K$ in $S^3$ admits no metric with positive scalar curvature and mean convex boundary. While we emphasize that these facts are strictly contained within known results about the topology of $3$-manifolds with positive scalar curvature and mean convex boundary (cf. Carlotto and Li's recent paper \cite{CarLi} and references therein), we believe that the straightforward proof via harmonic one-forms may be of some independent interest.
\end{remark}

Finally, in Section \ref{closed.sec}, we turn our attention back to the setting of closed $3$-manifolds, and extend the estimate \eqref{og.est} to any closed, oriented $3$--manifold. 

\begin{theorem}\label{thurst.thm.2} Let $(M^3,g)$ be a closed, oriented three-manifold. Then for any class $\alpha\in H_2(M;\mathbb{Z})$, we have
\begin{equation}
4\pi \|\alpha\|_{Th}\leq \|\alpha\|_H\|R_g^-\|_{L^2}.
\end{equation}
Moreover, if $\alpha$ cannot be represented by spheres, then equality implies that $M$ is covered isometrically by a cylinder $\Sigma \times \mathbb{R}$ over a closed surface $\Sigma$ of constant nonpositive curvature.
\end{theorem}

\begin{remark}
As in \cite{KM} and \cite{St19}, it is easy to find a family of metrics for which the inequality \eqref{og.est} approaches equality, so that
$$\|\alpha\|_{Th}=\frac{1}{4\pi}\inf\{\|\alpha\|_H\|R_g^-\|_{L^2}\mid g\in Met(M)\}$$
gives a geometric characterization of the Thurston norm on arbitrary closed $3$-manifolds. 
\end{remark}

The proof of Theorem \ref{thurst.thm.2} combines the result of Corollary \ref{mean.cvx.cor} with some deep results from minimal surface theory--namely, the geometric sphere theorem of Meeks and Yau \cite{MY}. Roughly, we apply the work of Meeks and Yau to find a disjoint collection $S_1,\ldots, S_k$ of embedded, nonseparating minimal spheres generating the portion of $H_2(M;\mathbb{Z})$ that can be represented entirely by embedded spheres, and apply Corollary \ref{mean.cvx.cor} to the manifold $N$ with $2k$ minimal boundary components obtained by cutting $M$ along these spheres. After arguing that the harmonic norm is nonincreasing under the natural map $\Phi: H_2(M;\mathbb{Z})\to H_2(N,\partial N;\mathbb{Z})$, and the Thurston norm is nondecreasing, we appeal to Corollary \ref{mean.cvx.cor} to arrive at the desired conclusion.

Among other consequences, Theorem \ref{thurst.thm.2} provides a new perspective on the fact--originally proved by Schoen and Yau in \cite{SY.psc.1}--that the homology $H_2(M;\mathbb{Z})$ of any closed $3$-manifold of positive scalar curvature is generated entirely by embedded spheres. It is interesting that, while our proof of Theorem \ref{thurst.thm.2} also relies on minimal surface theory, our use of minimal surfaces is somewhat orthogonal to that of Schoen and Yau, using minimal representatives of spherical homology classes to produce suitable boundary conditions for our harmonic form techniques.

In particular, our methods give an alternative technique for ruling out the existence of positive scalar curvature metrics on any manifold of the form $M\approx T^3 \# M_0$--a fact which, by an observation of Lohkamp \cite{Loh}, implies the Riemannian positive mass theorem in dimension three. In a separate paper \cite{BKKS} with Demetre Kazaras and Marcus Khuri, we apply variants of these techniques in the asymptotically flat setting to obtain a more direct proof of the positive mass theorem, providing a family of lower bounds for the ADM mass which are nonnegative on manifolds of nonnegative scalar curvature, and vanish precisely when certain harmonic coordinates are parallel.

\section*{acknowledgements}
The authors wish to thank Chao Li, Yevgeny Liokumovich, Demetre Kazaras, and Marcus Khuri for stimulating conversations related to this work. During the completion of this work, DS was supported in part by the NSERC grants of Robert Haslhofer, Robert Jerrard, and Yevgeny Liokumovich, whose support he gratefully acknowledges.

\section{Derivation of main identity}\label{main.comp}

In this section, we lay out carefully the computations leading to Theorem \ref{id.thm}. The argument is essentially the same as that of \cite{St19} Section 2, now with an examination of the boundary terms for $u$ satisfying the homogeneous Neumann condition \eqref{neu}. We remark that related computations can also be found in the paper \cite{BKKS} with Kazaras and Khuri.

\begin{proof}[Proof of Theorem \ref{id.thm}]

Let $u: N\to S^1$ be harmonic \eqref{harm} with homogeneous Neumann condition \eqref{neu} on $\partial N$, and denote by $\varphi_{\delta}$ the perturbation
$$\varphi_{\delta}:=(|h|^2+\delta)^{1/2}$$
of the norm $|h|=|du|$ of the gradient one-form $h=u^*(d\theta)$. As discussed in \cite{St19}, a simple application of the Schoen-Yau rearrangement trick to the Ricci term in the Bochner identity for $\Delta |h|^2$ yields the bound
\begin{equation}\label{lap.bd}
\Delta \varphi_{\delta}\geq\frac{|h|}{2\varphi_{\delta}}(|h|^{-1}|Dh|^2+R_N-R_{\Sigma})
\end{equation}
away from critical points of $u$, where $R_N$ denotes the scalar curvature of $N$ and $R_{\Sigma}$ denotes the intrinsic scalar curvature of the level set $\Sigma_{\theta}=u^{-1}\{\theta\}$ passing through a given point.

Moreover, away from critical points of $u$ along the boundary $\partial N$, we see that
\begin{eqnarray*}
\langle d\varphi_{\delta},\nu\rangle &=&\varphi_{\delta}^{-1}\langle \frac{1}{2}d|h|^2,\nu\rangle\\
&=&\varphi_{\delta}^{-1}\langle D_hh,\nu\rangle\\
&=&-\varphi_{\delta}^{-1}\langle h,D_h\nu\rangle,
\end{eqnarray*} 
where $\nu$ is the outward unit normal to $\partial N$, and in the last line we've used the Neumann condition \eqref{neu} to conclude that $\langle D_hh,\nu\rangle+\langle h,D_h\nu\rangle=0$.

Now, let $A\subset S^1$ be an open set containing the critical values of $u|_N$, and let $B\subset S^1$ be the complementary closed subset of regular values. Note that, by virtue of the Neumann condition \eqref{neu}, $A$ contains the critical values of the restriction $u|_{\partial N}$ as well.
For $\theta\in B$, denote by $Q_{\delta}(\theta)$ the quantity
$$Q_{\delta}(\theta):=\int_{\Sigma_{\theta}}\frac{|h|}{2\varphi_{\delta}}(|h|^{-2}|Dh|^2+R_N-R_{\Sigma})+\int_{\partial\Sigma_{\theta}}\varphi_{\delta}^{-1}\langle \frac{h}{|h|},D_h\nu\rangle.$$

Combining the coarea formula with the preceding computations, we then see that
\begin{eqnarray*}
\int_{\theta \in B}Q_{\delta}(\theta)&\leq &\int_{u^{-1}(B)}\Delta \varphi_{\delta}+\int_{u^{-1}(B)\cap \partial N}\varphi_{\delta}^{-1}\langle h,D_h\nu\rangle\\
&=&\int_{u^{-1}(B)}\Delta \varphi_{\delta}-\int_{u^{-1}(B)\cap \partial N}\langle d\varphi_{\delta},\nu\rangle\\
&=&-\int_{u^{-1}(A)}\Delta \varphi_{\delta}+\int_{u^{-1}(A)\cap \partial N}\langle d\varphi_{\delta},\nu\rangle,
\end{eqnarray*}
where we have used the Neumann condition \eqref{neu} to see that $|d(u|_{\partial N})|=|du|=|h|$ when applying the coarea formula on $\partial N$. In particular, since 
$$\Delta \varphi_{\delta}\geq |Ric_N| |h|\geq -C|h|$$
globally on $N$ by the standard form of the Bochner identity, and
$$\langle d\varphi_{\delta},\nu\rangle\leq C |h|,$$
applying the coarea formula once more in the preceding computation gives
\begin{equation}\label{q.bd}
\int_{\theta\in B} Q_{\delta}(\theta)\leq C\int_{\theta\in A}[\mathcal{H}^2(\Sigma_{\theta})+\mathcal{H}^1(\partial \Sigma_{\theta})].
\end{equation}

Now, setting $\eta=h/|h|$, as $\delta\to 0$, we see that $Q_{\delta}$ converges uniformly on the compact set $B$ of regular values to
$$Q(\theta)=\int_{\Sigma_{\theta}}\frac{1}{2}(|h|^{-2}|Dh|^2+R_N-R_{\Sigma})+\int_{\partial\Sigma_{\theta}}\langle \eta, D_{\eta}\nu\rangle.$$
Next, by the Neumann condition \eqref{neu}, note that the normal $\nu$ to $\partial N$ is also the outward unit normal to the curve $\partial\Sigma_{\theta}$ as the boundary of $\Sigma_{\theta}$, so that the cross product
$$\tau:=*(\eta\wedge \nu),$$
gives a unit tangent to $\partial \Sigma_{\theta}$, and the geodesic curvature $\kappa_{\partial\Sigma_{\theta}}$ is given by
$$\kappa_{\partial \Sigma}=\langle \tau,D_{\tau}\nu\rangle.$$
On the other hand, since $\eta,\tau$ is an orthonormal basis for $\partial N$, we see that
$$H_{\partial N}=\langle \eta,D_{\eta}\nu\rangle+\langle \tau,D_{\tau}\nu\rangle,$$
and putting all this together, we find that
\begin{eqnarray*}
Q(\theta)&=&\int_{\Sigma_{\theta}}\frac{1}{2}(|h|^{-2}|Dh|^2+R_N-R_{\Sigma})+\int_{\partial \Sigma_{\theta}}(H_{\partial N}-\kappa_{\partial\Sigma})\\
&=&\int_{\Sigma_{\theta}}\frac{1}{2}(|h|^{-2}|Dh|^2+R_N)+\int_{\partial\Sigma_{\theta}}H_{\partial N}-2\pi \chi(\Sigma_{\theta}),
\end{eqnarray*}
where in the last line we have applied the Gauss-Bonnet formula for $\Sigma_{\theta}$. Passing to the limit $\delta \to 0$ in \eqref{q.bd}, we see now that
\begin{eqnarray*}
\int_{\theta\in B}\left(\int_{\Sigma_{\theta}}\frac{1}{2}(|h|^{-2}|Dh|^2+R_N)+\int_{\partial \Sigma_{\theta}}H_{\partial N}\right)&\leq &\int_{\theta\in B}2\pi \chi(\Sigma_{\theta})\\
&&+C\int_{\theta\in A}(\mathcal{H}^2(\Sigma_{\theta})+\mathcal{H}^1(\partial\Sigma_{\theta})).
\end{eqnarray*}
Finally, by Sard's theorem, we can take the measure of $A$ arbitrarily small, and since $\theta\mapsto \mathcal{H}^2(\Sigma_{\theta})+\mathcal{H}^1(\partial \Sigma_{\theta})$ is an integrable function on $S^1$ (by the coarea formulas for $u|_N$ and $u|_{\partial N}$), it follows that
\begin{equation}\label{main.id}
\int_{\theta\in S^1}\left(\int_{\Sigma_{\theta}}\frac{1}{2}(|h|^{-2}|Dh|^2+R_N)+\int_{\partial\Sigma_{\theta}}H_{\partial N}\right)\leq \int_{\theta\in S^1}2\pi \chi(\Sigma_{\theta}).
\end{equation}

\end{proof}

\section{Scalar curvature, mean curvature, and the Thurston norm for manifolds with boundary}\label{bdry.thurst.sec}

With the identity \eqref{main.id} in place, we move on now to the proof of Theorem \ref{bdry.thurst}. The proof is fairly straightforward, following much the same lines as that of Theorem 1.2 in \cite{St19}.

\begin{proof}[Proof of Theorem \ref{bdry.thurst}] Fix a nontrivial relative homology class $\alpha\in H_2(N,\partial N;\mathbb{Z})$, and let $u: N\to S^1$ be an energy-minimizing representative of the dual homotopy class, so that the gradient one-form $h=u^*(d\theta)$ satisfies \eqref{harm} and \eqref{neu}. For such a map $u$, we observe that every connected component $\Gamma$ of a regular level set $\Sigma_{\theta}$ defines a nontrivial class in $H_2(N,\partial N)$: this follows from the simple observation that $\Gamma$ has positive pairing
$$\int_{\Gamma} *h=\int_{\Gamma}|h|>0,$$
with the $2$-form $*h$ Hodge dual to $h$--which, by virtue of \eqref{harm}-\eqref{neu}, satisfies
$$d(*h)\equiv 0\text{ on }N\text{ and }\iota^*_{\partial N}(*h)=0\text{ on }\partial N.$$
In particular, it follows from the hypotheses of the theorem that every connected component of a regular level set $\Sigma_{\theta}$ has nonpositive Euler characteristic, so that
\begin{equation}\label{th.bd}
\|\alpha\|_{Th}\leq \chi_-(\Sigma_{\theta})=-\chi(\Sigma_{\theta}),
\end{equation}
by definition \eqref{th.def} of the Thurston norm.

As a consequence, integrating \eqref{th.bd} over $S^1$ and appealing to Theorem \ref{id.thm}, we see that
\begin{eqnarray*}
2\pi \|\alpha\|_{Th}&\leq &-\int_{\theta\in S^1}2\pi \chi(\Sigma_{\theta})\\
&\leq & -\int_{\theta\in S^1}\left(\int_{\Sigma_{\theta}}\frac{1}{2}(|du|^{-2}|Hess(u)|^2+R_N)+\int_{\partial\Sigma_{\theta}}H_{\partial N}\right)\\
&\leq &-\frac{1}{2}\int_N |du|R_N-\int_{\partial N}|du| H_{\partial N},
\end{eqnarray*}
with equality only if $Hess(u)\equiv 0$. Applying the Cauchy-Schwarz inequality, it follows immediately that
$$2\pi \|\alpha\|_{Th}\leq \frac{1}{2}\|du\|_{L^2(N)}\|R_N^-\|_{L^2}+\|du\|_{L^2(\partial N)}\|H_{\partial N}^-\|_{L^2(\partial N)},$$
giving us the desired estimate, with equality only if $Hess(u)\equiv 0$, and $R_N\equiv -c_1|du|$, $H_{\partial N}\equiv -c_2|du|$ are nonpositive constants. 

Finally, if equality holds, then since $Hess(u)\equiv 0$ and $u$ satisfies the Homogeneous Neumann condition \eqref{neu}, then fixing any component $S$ of a regular level set $\Sigma=u^{-1}\{\theta\}$, we see that the gradient flow 
$$\Phi: S\times \mathbb{R}\to N,\text{ }\frac{\partial \Phi}{\partial t}=\frac{grad(u)}{|grad(u)|}\circ \Phi$$
defines a local isometry. That $\Sigma$ has constant curvature and $\partial \Sigma$ has constant geodesic curvature than follows from the constancy of $R_N$ and $H_{\partial N}$, respectively.
\end{proof}

The path from Theorem \ref{bdry.thurst} to Corollary \ref{mean.cvx.cor} for manifolds with mean convex boundary is immediate. In the case of equality, the rigidity statement in Theorem \ref{bdry.thurst} together with the mean convexity $H_{\partial N}\geq 0$ evidently implies that the boundary $\partial N$ is a union of totally geodesic tori.

\section{Scalar curvature and Thurston norm for general closed $3$-manifolds}\label{closed.sec}

In this section, we employ Corollary \ref{mean.cvx.cor} together with a cutting procedure to prove Theorem \ref{thurst.thm.2}, thereby generalizing Theorem 2 of \cite{KM} and Theorem 1.2 of \cite{St19} to arbitrary closed, oriented $3$-manifolds. Recall that the case left open by Theorem 1.2 of \cite{St19} is that of $3$-manifolds containing nonseparating spheres--that is, manifolds whose prime decomposition contains an $S^2\times S^1$ factor.

To this end, fix an arbitrary closed, oriented $3$-manifold $(M^3,g)$, and consider the subgroup $\mathcal{S}\subset H_2(M;\mathbb{Z})$ generated by embedded $2$-spheres in $M$; that is, let $\mathcal{S}$ be the maximal subset of $H_2(M;\mathbb{Z})$ consisting of homology classes representable by unions of spheres. The key first step in our argument consists of finding a collection of disjoint collection of minimal $2$-spheres generating $\mathcal{S}$, by appealing to the work of Meeks and Yau \cite{MY}.

\begin{lemma}\label{mksyau.lem} There exists a finite collection of disjoint embedded, nonseparating, minimal two-spheres
$$S_1,\ldots,S_k\subset M$$
generating $\mathcal{S}$.
\end{lemma}

\begin{proof} By Theorems 5 and 7 of \cite{MY} (see also Section 4 of \cite{HS}), we know that there exists a finite collection 
$$f_1,\ldots,f_m: S^2\to M$$
of conformal minimal immersions which generate $\pi_2(M)$ as a $\pi_1(M)$-module. Moreover, each map $f_i$ is either an embedding of $S^2$ or factors through an embedding of $\mathbb{RP}^2$, and the images $f_i(M)$ are disjoint. 

Of these maps $f_i$, consider the subcollection $f_1,\ldots,f_k$ for which 
\begin{equation}\label{hom.cond}
(f_i)_*[S^2]\neq 0\in H_2(M;\mathbb{Z}).
\end{equation}
Since the full collection $f_1,\ldots,f_m$ generates $\pi_2(M)$, it's clear that the pushforwards $S_1=(f_1)_*(S^2),\ldots, S_k=(f_k)_*(S^2)$ must generate $\mathcal{S}$ in $H_2(M;\mathbb{Z})$. Moreover, the nontriviality of the pushforward $S_i=(f_i)_*[S^2]$ in integral homology implies that $f_i$ \emph{does not} factor through a map $\mathbb{RP}^2\to M$, so each $S_i$ must correspond to an embedded, two-sided $2$-sphere in $M$. Finally, since we know already that the maps $f_i$ have disjoint images, it follows that $S_1,\ldots,S_k$ satisfy the desired properties.
\end{proof}

Now, consider the compact manifold with boundary $N$ obtained by cutting $M$ along the spheres $S_1,\ldots,S_k$ of Lemma \ref{mksyau.lem}. Note that the boundary $\partial N$ of $N$ consists of $2k$ minimal spheres, and that $N$ carries an obvious local isometry
$$\Psi: N\to M$$
which restricts to a global isometry $N\setminus \partial N\to M\setminus \bigcup_{j=1}^kS_k$. Our proof of Theorem \ref{thurst.thm.2} rests on an analysis of the map
$$\Phi=\Psi^*: H^1(M;\mathbb{Z})\to H^1(N;\mathbb{Z})$$
given by pullback under $\Psi$; equivalently, we may view $\Phi$ via Poincar\'{e}-Lefschetz duality as a map
$$\Phi: H_2(M;\mathbb{Z})\to H_2(N,\partial N;\mathbb{Z}).$$
Since $\Psi:N\to M$ restricts to a global isometry on the interior of $N$, it is easy to see that $\Phi$ \emph{decreases} the harmonic norm.

\begin{lemma}\label{phi.harm} For every $\alpha\in H^1(M;\mathbb{Z})$, we have
\begin{equation}
\|\Phi(\alpha)\|_H\leq \|\alpha\|_H.
\end{equation}
\end{lemma}

\begin{proof} Given $\alpha\in H^1(M;\mathbb{Z})\cong [M:S^1]$, let $u: M\to S^1$ be the harmonic representative, so that
$$\|\alpha\|_H=\|du\|_{L^2(M)}.$$
The harmonic norm $\|\Phi(\alpha)\|_H$ of the class $\Phi(\alpha)\in H^1(N;\mathbb{Z})$ is then given by the infimum of the $L^2$ norm of the gradient among maps homotopic to $u\circ \Psi$ in $N$. In particular, it follows that
$$\|\Phi(\alpha)\|_H\leq \|d(u\circ \Psi)\|_{L^2(N)}=\|du\|_{L^2(M)}=\|\alpha\|_H,$$
as claimed.
\end{proof}

Next, we argue that the map $\Phi: H_2(M;\mathbb{Z})\to H_2(N,\partial N;\mathbb{Z})$ \emph{does not decrease} the Thurston norm of a class $\alpha\in H_2(M;\mathbb{Z})$. In the course of the proof, we make repeated use of the following simple topological lemma, which allows us to remove boundary components from the representatives of classes in $H_2(N,\partial N;\mathbb{Z})$.

\begin{lemma}\label{close.lem} For any connected, embedded surface $(\Sigma,\partial \Sigma)\subset (N,\partial N)$, there exists a closed surface $\tilde{\Sigma}\subset N$ of the same genus such that $[\Sigma]=[\tilde{\Sigma}]\in H_2(N,\partial N;\mathbb{Z})$.
\end{lemma}

\begin{proof} Consider a connected, embedded $(\Sigma,\partial \Sigma)\subset (N,\partial N)$, and assume without loss of generality that $\Sigma$ meets $\partial N$ transversally along $\partial\Sigma$. Suppose $\partial \Sigma\neq \varnothing$ (otherwise take $\tilde{\Sigma}=\Sigma$), so that $\partial \Sigma$ consists of some number $\ell$ of embedded loops $\gamma_1,\ldots,\gamma_{\ell}$ in $\partial N$. 

Since $\partial N$ consists of spheres, it follows that each boundary component $\gamma_i$ of $\Sigma$ may be realized as the boundary of a topological disk $D_i\subset \partial N$. Moreover, since the distinct boundary components $\gamma_i$ do not intersect, it follows that for every pair $D_i,D_j$ of these disks, either $D_i\cap D_j=\varnothing$, or one disk is strictly contained in the other. In particular, it follows that we can find at least one disk--which, after relabeling, we call $D_{\ell}$--with the property that for every $i<\ell$, either $D_{\ell}\subset D_i$ or $D_{\ell}\cap D_i=\varnothing$.

Attaching this disk $D_{\ell}$ to $\Sigma$, we obtain a new (piecewise smooth) surface $\Sigma'$ such that $spt(\Sigma-\Sigma')\subset \partial N$, with genus $g(\Sigma')=g(\Sigma)$ and $\ell-1$ boundary components. We may then perturb $\Sigma'$ by a homotopy to obtain an embedded surface $\tilde{\Sigma}_{\ell}\subset N$ homologous to $\Sigma'$, meeting $\partial N$ transverally along $\partial \Sigma_{\ell}$. Repeating this process, we obtain a sequence $\tilde{\Sigma}_{\ell},\ldots,\tilde{\Sigma}_1$, terminating at a closed surface $\tilde{\Sigma}=\tilde{\Sigma}_1$ with genus $g(\tilde{\Sigma})=g(\Sigma)$, such that $[\tilde{\Sigma}]=[\Sigma]\in H_2(N,\partial N;\mathbb{Z})$.
\end{proof}

Armed with Lemma \ref{close.lem}, we can now easily prove that the map $\Phi: H_2(M;\mathbb{Z})\to H_2(N,\partial N;\mathbb{Z})$ does not decrease the Thurston norm.

\begin{claim}\label{phi.thurst} For every $\alpha\in H_2(M;\mathbb{Z})$, we have 
\begin{equation}
\|\Phi(\alpha)\|_{Th}\geq\|\alpha\|_{Th}.
\end{equation}
\end{claim}

\begin{proof} To see this, consider an embedded surface $\Sigma\subset N$ representing $\Phi(\alpha)\in H_2(N,\partial N;\mathbb{Z})$, such that
$$\chi_-(\Sigma)=\|\Phi(\alpha)\|_{Th}.$$
Applying the construction of Lemma \ref{close.lem} to each connected component of $\Sigma$ (an operation which evidently decreases $\chi_-$), we may assume moreover that $\Sigma$ is closed, and does not intersect the boundary spheres $\partial N$. 

Finally, note that we can produce a surface $\Gamma\subset M$ representing the class $\alpha\in H_2(M;\mathbb{Z})$ by adding some finite combination of the spheres $S_1,\ldots, S_k$ to the closed representative $\Sigma\subset N\setminus \partial N=M\setminus \bigcup_{j=1}^kS_j$ of $\Phi(\alpha)$. Since the spheres make no contribution to $\chi_-(\Gamma)$, we then see that
$$\|\alpha\|_{Th}\leq \chi_-(\Gamma)=\chi_-(\Sigma)=\|\Phi(\alpha)\|_{Th},$$
as desired.
\end{proof}

\begin{remark}\label{no.gen.0} As another consequence of Lemma \ref{close.lem}, we observe that any surface of genus zero in $N$ defines a trivial class in $H_2(N,\partial N;\mathbb{Z})$. Indeed, given an embedded surface $\Sigma \subset N$ of genus zero, by Lemma \ref{close.lem}, we may always find an embedded sphere $\tilde{\Sigma}\subset N$ homologous to $\Sigma$ in $H_2(N,\partial N)$. On the other hand, since the spheres $S_1,\ldots, S_k$ generate $\mathcal{S}$, it follows that this sphere $\tilde{\Sigma}$ must be homologous to some combination of $S_1,\ldots, S_k$ in $M$; hence $\tilde{\Sigma}$ is homologous to some combination of the boundary components of $\partial N$, and consequently
$$[\Sigma]=[\tilde{\Sigma}]=0\in H_2(N,\partial N;\mathbb{Z}).$$
\end{remark}

Combining the results of Lemma \ref{phi.harm}, Claim \ref{phi.thurst}, and the preceding remark with Corollary \ref{mean.cvx.cor}, the proof of Theorem \ref{thurst.thm.2} is now straightforward.

\begin{proof}[Proof of Theorem \ref{thurst.thm.2}]

Given a class $\alpha \in H_2(M;\mathbb{Z})$, consider its image $\Phi(\alpha)$ under the map $\Phi: H_2(M;\mathbb{Z})\to H_2(N,\partial N;\mathbb{Z})$ described above. By Remark \ref{no.gen.0}, we see that any surface of genus zero in $N$ induces a trivial cycle in $H_2(N,\partial N;\mathbb{Z})$, and since $N$ has minimal boundary, we may apply Corollary \ref{mean.cvx.cor} to conclude that
\begin{equation}\label{phi.bds}
4\pi \|\Phi(\alpha)\|_{Th}\leq \|\Phi(\alpha)\|_H\|R_N^-\|_{L^2(N)}=\|\Phi(\alpha)\|_H\|R_M^-\|_{L^2(M)}.
\end{equation}
Combining this relation with the estimates of Lemma \ref{phi.harm} and Claim \ref{phi.thurst} relating the harmonic and Thurston norms of $\Phi(\alpha)$ to those of $\alpha$, we immediately arrive at the desired bound
\begin{equation}\label{main.est}
4\pi \|\alpha\|_{Th}\leq \|\Phi(\alpha)\|_H\|R_M^-\|_{L^2(M)}.
\end{equation}

Finally, note that if $\alpha$ cannot be represented by spheres, then $\Phi(\alpha)\neq 0$, so if equality holds in \eqref{main.est}, then equality holds in \eqref{phi.bds} as well, and the rigidity statement in Corollary \ref{mean.cvx.cor} implies that $N$ is covered isometrically by a cylinder $\Sigma \times \mathbb{R}$ over a surface $\Sigma$ of constant nonpositive curvture with (possibly empty) geodesic boundary. Note, moreover, that if this surface $\Sigma$ has nonempty boundary, this would imply that the boundary $\partial N$ is made up of tori--which of course cannot be the case, since $\partial N$ consists of spheres by construction. Thus, if equality holds in \eqref{main.est}, it follows that $M=N$ is covered isometrically by the cylinder $\Sigma\times \mathbb{R}$
over a closed surface of constant nonpositive curvature, as claimed.

\end{proof}


\begin{thebibliography}{99}
\bibitem{BKKS} H. Bray, D. Kazaras, M. Khuri, and D. Stern. Harmonic functions and the mass of $3$-dimensional asymptotically flat Riemannian manifolds. (preprint)
\bibitem{CarLi} A. Carlotto and C. Li. Constrained deformations of positive scalar curvature metrics. 	arXiv:1903.11772
\bibitem{GL1} M. Gromov and H.B. Lawson, Jr. Spin and scalar curvature in the presence of a fundamental group I. \emph{Ann. Math.,} 111(2):209--230, 1980.
\bibitem{GL2} M. Gromov and H.B. Lawson, Jr. The classification of simply connected manifolds of positive scalar curvature. \emph{Ann. Math.,} 111(2):423--434, 1980.
\bibitem{GL3} M. Gromov and H.B. Lawson, Jr. Positive scalar curvature and the Dirac operator on complete Riemannian manifolds. \emph{Inst. Hautes Etudes Sci. Publ. Math.} 58:83--196, 1983.
\bibitem{HS} J. Hass and P. Scott. The existence of least area surfaces in $3$-manifolds. \emph{Trans. Amer. Math. Soc.} 310:87--114, 1988.
\bibitem{Hitch} N. Hitchin. Harmonic spinors. \emph{Adv. Math.} 14:1--55, 1974.
\bibitem{KM} P. Kronheimer and T. Mrowka. Scalar curvature and the Thurston norm. \emph{Math. Res. Lett.,} 4(6):931--937, 1997.
\bibitem{Lich} A. Lichnerowicz. Spineurs harmoniques. \emph{C.R. Acad. Sci. Paris.} 257: 7--9, 1963.
\bibitem{Loh} J. Lohkamp. Scalar curvature and hammocks. \emph{Math. Ann.} 313: 385--407, 1999.
\bibitem{MY} W. Meeks and S-T Yau. Topology of three dimensional manifolds and the embedding problems in minimal surface theory. \emph{Ann. Math.} 112: 441-484, 1980.
\bibitem{SY.psc.1} R. Schoen and S-T Yau. Existence of incompressible minimal surfaces and the topology of three--manifolds of positive scalar curvature. \emph{Ann. Math.,} 110(1):127--142, 1979.
\bibitem{SY.psc.2} R. Schoen and S-T Yau. On the structure of manifolds with positive scalar curvature. \emph{Manuscripta
mathematica,} 28:159--184, 1979.
\bibitem{SY.pmt} R. Schoen and S-T Yau. On the proof of the positive mass conjecture in
general relativity. \emph{Comm. Math. Phys.,} 65(1):45--76, 1979.
\bibitem{St19} D. Stern. Scalar curvature and harmonic maps to $S^1$. arXiv:1908.09754
\bibitem{Th} W. Thurston. A norm for the homology of $3$--manifolds. \emph{Mem. Amer. Math. Soc.,} 59(339):i-vi and 99--130, 1986.
\bibitem{Wit}  E. Witten. A simple proof of the positive energy theorem. \emph{Comm. Math. Phys.} 80(3):, 381-402, 1981.
\end{thebibliography}
\end{document}